\documentclass[a4paper,11pt]{article}
\usepackage[english]{babel}
\usepackage{amsmath,amssymb,stmaryrd,textcomp,bbm,xcolor,latexsym,amsthm,mathrsfs}
\usepackage[hidelinks]{hyperref}
\usepackage[top=3cm, bottom=3cm, left=2.5cm, right=2.5cm]{geometry}
\usepackage[makeroom]{cancel}

\newcommand{\assign}{:=}
\newcommand{\cdummy}{\cdot}
\newcommand{\mathD}{\mathrm{D}}
\newcommand{\mathd}{\mathrm{d}}
\newcommand{\nobracket}{}
\newcommand{\tmop}[1]{\ensuremath{\operatorname{#1}}}
\newcommand{\tmtextit}[1]{{\itshape{#1}}}
\newtheorem{theorem}{Theorem}

\newtheorem{definition}[theorem]{Definition}
\newtheorem{lemma}[theorem]{Lemma}
\newtheorem{remark}[theorem]{Remark}
\newtheorem{proposition}[theorem]{Proposition}

\newcommand{\cS}{\mathscr{S}}

\newcommand{\cL}{\mathscr{L}}

\newcommand{\dD}{\mathrm{D}}

\begin{document}

\title{The Hairer--Quastel universality result in equilibrium}
\author{
  Massimiliano Gubinelli\\
  Hausdorff Center for Mathematics\\
   \& Institute for Applied Mathematics\\
   Universit{\"a}t Bonn \\
  \texttt{gubinelli@iam.uni-bonn.de}
  \and
  Nicolas Perkowski \\
  Institut f\"ur Mathematik \\
  Humboldt--Universit\"at zu Berlin \\
  \texttt{perkowsk@math.hu-berlin.de}
}

\maketitle

\abstract{We use the notion of energy solutions of the stochastic Burgers equation to give a short proof of the Hairer-Quastel universality result for a class of stationary weakly asymmetric stochastic PDEs.}

\section{Introduction}

Consider the stochastic PDE
\begin{equation}
  \label{eq:weakly-asym} \partial_t v = \Delta v + \varepsilon^{1 / 2}
  \partial_x F (v) + \partial_x \chi^{\varepsilon}
\end{equation}
on $[0, \infty) \times \mathbb{T}_{\varepsilon}$ with
$\mathbb{T}_{\varepsilon} =\mathbb{R}/ (2 \pi \varepsilon^{- 1}
\mathbb{Z})$, where $\chi^{\varepsilon}$ is a Gaussian noise that is white in
time and spatially smooth. The celebrated Hairer--Quastel universality result~\cite{HQ15} states that there exist constants $c_1, c_2 \in
\mathbb{R}$ such that the rescaled process $\varepsilon^{- 1 / 2} v_{t
\varepsilon^{- 2}} ((x - c_1 \varepsilon^{- 1 / 2} t) \varepsilon^{- 1})$
converges to the solution $u$ of the stochastic Burgers equation
\[ \partial_t u = \Delta u + c_2 \partial_x u^2 + \xi, \]
where $\xi$ is a space-time white noise. Here we give an alternative proof of
this result, based on the concept of energy solutions~\cite{goncalves_nonlinear_2013,GJ13,gubinelli_lectures_2015,GP15}. Energy solutions formulate the equilibrium
Burgers equation as a martingale problem and allow us to give a simpler proof than the one of~\cite{HQ15}. On the other side our method only
applies in equilibrium and in fact at each step we need to know the invariant
measure explicitly.

Let us state the result more precisely. We modify~(\ref{eq:weakly-asym}) such
that after rescaling $\tilde{u}^{\varepsilon}_t (x) = \varepsilon^{- 1 / 2}
v_{t \varepsilon^{- 2}} (x \varepsilon^{- 1})$ we have
\begin{equation}
  \label{eq:HQ-start} \partial_t \tilde{u}^{\varepsilon} = \Delta
  \tilde{u}^{\varepsilon} + \varepsilon^{- 1} \partial_x \Pi_0^N F
  (\varepsilon^{1 / 2} \tilde{u}^{\varepsilon}) + \partial_x \Pi_0^N
  \tilde{\xi}, \hspace{2em} \tilde{u}^{\varepsilon}_0 = \Pi_0^N \eta,
\end{equation}
where $\tilde{\xi}$ is a space-time white noise on $[0, \infty) \times
\mathbb{T}$ (where $\mathbb{T}=\mathbb{T}_1$) with variance 2, $\eta$ is a
space white noise which is independent of $\tilde{\xi}$, $\Pi_0^N$ denotes the
projection onto the Fourier modes $0 < | k | \leqslant N$, and we always link
$N$ and $\varepsilon$ via
\[ 2 N = 1/\varepsilon . \]
\begin{theorem}
  \label{thm:main-result}Let $F$ be almost everywhere differentiable and assume that for all
  $\varepsilon > 0$ there is a unique solution $\tilde{u}^{\varepsilon}$
  to~(\ref{eq:HQ-start}) which does not blow up before $T > 0$. Assume also
  that $F, F' \in L^2 (\nu)$ where $\nu$ is the standard normal distribution.
  Then $u^{\varepsilon}_t (x) \assign \tilde{u}^{\varepsilon}_t (x -
  \varepsilon^{- 1 / 2} c_1 (F) t \nobracket)$, $(t, x) \in [0, T] \times
  \mathbb{T}$, converges in distribution to the unique equilibrium energy
  solution $u$ of
  \[ \partial_t u = \Delta u + c_2 (F) \partial_x u^2 + \xi, \]
  where $\xi$ is a space-time white noise with variance $2$ and for $U \sim \nu$ and $k
  \geqslant 0$ and $H_k$ the $k$-th Hermite polynomial
  \[ c_k (F) = \frac{1}{k!} \mathbb{E} [F (U) H_k (U)] . \]
\end{theorem}

\begin{remark}
   If $F$ is even, then $c_1(F) = 0$ while $c_2(F)=0$ if $F$ is odd.
\end{remark}

\begin{remark}
  Note that we introduced a second regularization in~(\ref{eq:HQ-start})
  compared to~(\ref{eq:weakly-asym}) which acts on $F(\varepsilon^{1/2}u^\varepsilon)$. The reason is that we need to keep track of
  the invariant measure and this second regularization allows us to write it down explicitly. For the moment we are unable to deal with the original equation~(\ref{eq:weakly-asym}). For simplicity
  here we only consider the mollification operator $\Pi_0^N$, but it is
  possible to extend everything to more general operators $\rho (\varepsilon \mathD)
  u =\mathcal{F}^{- 1} (\rho (\varepsilon \cdummy) \mathcal{F}u)$, where
  $\mathcal{F}$ denotes the Fourier transform and $\rho$ is an even, compactly
  supported, bounded function which is continuous in a neighborhood of $0$ and
  satisfies $\rho (0) = 1$. We should then modify the equation as
  \[ \partial_t \tilde{u}^{\varepsilon} = \Delta u^{\varepsilon} +
     \varepsilon^{- 1} \partial_x \rho (\varepsilon \mathD) \rho (\varepsilon
     \mathD) F (\varepsilon^{1 / 2} \tilde{u}^{\varepsilon}) + \partial_x \rho
     (\varepsilon \mathD) \tilde{\xi}, \hspace{2em} \tilde{u}^{\varepsilon}_0
     = \rho (\varepsilon \mathD) \eta, \]
  to keep control of the invariant measure.
\end{remark}

\begin{remark}
   While our result only applies in equilibrium, we have more freedom in choosing the nonlinearity $F$ than~\cite{HQ15} who require it to be an even polynomial. Also, the methods of this paper will extend without great difficulty to the (modified) equation on $[0,T] \times \mathbb R$.
\end{remark}

\paragraph{Notation}For $k \in \mathbb{Z}$ we write $e_k (x) = e^{i k
x}/\sqrt{2 \pi}$ for the $k$-th Fourier monomial, and for $u \in \cS'$, the
distributions on $\mathbb{T}$, we define $\hat{u} (k) =\mathcal{F}u (k) =
\langle u, e_{- k} \rangle$. We use $\langle \cdummy, \cdummy \rangle$ to
denote both the duality pairing in $\cS' \times C^{\infty} (\mathbb{T},
\mathbb{C})$ and the inner product in $L^2 (\mathbb{T})$, so since we want
the notation to be consistent we will always consider the $L^2 (\mathbb{T},
\mathbb{R})$ inner product and not that of $L^2 (\mathbb{T}, \mathbb{C})$.
That is, even for complex valued $f, g$ we set $\langle f, g \rangle =
\int_{\mathbb{T}} f (x) g (x) \mathd x$ and do not take a complex conjugate.
The Fourier projection operator $\Pi_0^N$ is given by
\[ \Pi_0^N v = \sum_{0 < | k | \leqslant N} e_k \hat{v} (k) . \]

\section{Preliminaries}

Let us start by making some basic observations concerning the solution
to~(\ref{eq:HQ-start}).

\paragraph{Galilean transformation}Recall that $\tilde{u}^{\varepsilon}$
solves
\[ \partial_t \tilde{u}^{\varepsilon} = \Delta \tilde{u}^{\varepsilon} +
   \varepsilon^{- 1} \partial_x \Pi_0^N F (\varepsilon^{1 / 2}
   \tilde{u}^{\varepsilon}) + \partial_x \Pi_0^N \tilde{\xi}, \]
and that $u^{\varepsilon}_t (x) = \tilde{u}^{\varepsilon}_t (x -
\varepsilon^{- 1 / 2} c_1 (F) t)$. We define the modified test function
$\tilde{\varphi}_t (x) = \varphi (x + \varepsilon^{- 1 / 2} c_1 (F) t)$ and
then $\langle u^{\varepsilon}_t, \varphi \rangle = \langle
\tilde{u}^{\varepsilon}_t, \tilde{\varphi}_t \rangle$. The It™\^o--Wentzell
formula gives
\begin{align*}
   \mathd \langle u^{\varepsilon}_t, \varphi \rangle & = \langle \mathd
   \tilde{u}^{\varepsilon}_t, \tilde{\varphi}_t \rangle + \langle
   \tilde{u}^{\varepsilon}_t, \partial_t \tilde{\varphi}_t \rangle \mathd t \\
   & = \langle \Delta \tilde{u}^{\varepsilon}_t, \tilde{\varphi}_t \rangle
   \mathd t + \langle \varepsilon^{- 1} \partial_x \Pi_0^N F (\varepsilon^{1 /
   2} \tilde{u}^{\varepsilon}), \tilde{\varphi}_t \rangle \mathd t + \langle
   \mathd \partial_x \tilde{M}^{\varepsilon}_t, \tilde{\varphi}_t \rangle \\
   & \qquad + \langle \varepsilon^{- 1 / 2} c_1 (F) \tilde{u}^{\varepsilon}_t, \partial_x
   \tilde{\varphi}_t \rangle \mathd t,
\end{align*}
where $\tilde{M}^{\varepsilon}_t (x) = \int_0^t \Pi_0^N \tilde{\xi} (s, x)
\mathd s$. Integrating the last term on the right hand side by parts, we get
\[ \mathd \langle u^{\varepsilon}_t, \varphi \rangle = \langle \Delta
   u^{\varepsilon}_t, \varphi \rangle \mathd t + \langle \varepsilon^{- 1}
   \partial_x \Pi_0^N F (\varepsilon^{1 / 2} u^{\varepsilon}), \varphi \rangle
   \mathd t - \varepsilon^{- 1 / 2} c_1 (F) \langle \partial_x
   u^{\varepsilon}_t, \varphi \rangle \mathd t + \langle \mathd \partial_x
   \tilde{M}^{\varepsilon}_t, \tilde{\varphi}_t \rangle . \]
The martingale term has quadratic variation
\[ \mathd [\langle \partial_x \tilde{M}^{\varepsilon}, \tilde{\varphi}_t
   \rangle]_t = \mathd [\langle \tilde{M}^{\varepsilon}, \partial_x
   \tilde{\varphi}_t \rangle]_t = 2 \| \Pi_0^N \partial_x \tilde{\varphi}_t
   \|_{L^2}^2 \mathd t = 2 \| \Pi_0^N \partial_x \varphi \|_{L^2}^2 \mathd t,
\]
which means that the process $\langle M^{\varepsilon}_t, \varphi \rangle
\assign \langle \tilde{M}^{\varepsilon}_t, \tilde{\varphi}_t \rangle$ is of
the form $M^{\varepsilon}_t = \int_0^t \Pi_0^N \xi (s, x) \mathd s$ for a new
space-time white noise $\tilde{\xi}$ with variance 2. In conclusion,
$u^{\varepsilon}$ solves
\begin{equation}
  \label{eq:HQ-transformed} \partial_t u^{\varepsilon} = \Delta
  u^{\varepsilon} + \varepsilon^{- 1} \partial_x \Pi_0^N (F (\varepsilon^{1 /
  2} u^{\varepsilon}) - c_1 (F) \varepsilon^{1 / 2} \partial_x
  u^{\varepsilon}) + \partial_x \Pi_0^N \xi, \hspace{2em} u^{\varepsilon}_0 =
  \Pi_0^N \eta,
\end{equation}
so in other words by performing the change of variables $u^{\varepsilon}_t (x)
= \tilde{u}^{\varepsilon}_t (x - \varepsilon^{- 1 / 2} c_1 (F) t)$ we replaced
the function $F$ by $\tilde F(x) = F (x) - c_1 (F) x$, and now it suffices to study
equation~(\ref{eq:HQ-transformed}).

\paragraph{Invariant measure}Note that~(\ref{eq:HQ-transformed}) actually is
an SDE in the finite dimensional space $Y_N = \Pi_0^N L^2 (\mathbb{T},
\mathbb{R}) \simeq \mathbb{R}^{2 N}$, so that we can apply Echeverria's
criterion to show the stationarity of a given distribution. The natural
candidate is ${\mu}^{\varepsilon} = \tmop{law} (\Pi_0^N \eta)$, where
$\eta$ is a space white noise, since we know that the dynamics of the
regularized Ornstein-Uhlenbeck process
\[ \partial_t X^{\varepsilon} = \Delta X^{\varepsilon} + \partial_x \Pi_0^N
   \xi \]
are invariant and even reversible under ${\mu}^{\varepsilon}$ and that for models in the KPZ universality class the asymmetric version often has the same invariant measure as the symmetric one. Let us
write
\[ B_F^{\varepsilon} (u) = \varepsilon^{- 1} \partial_x \Pi^N_0 (F
   (\varepsilon^{1 / 2} u) - c_1 (F) \varepsilon^{1 / 2} u) = : \varepsilon^{-
   1} \partial_x \Pi^N_0 \tilde{F} (\varepsilon^{1 / 2} u), \]
where $\tilde{F} = F - c_1(F) x$.

\begin{lemma}
  The vector field $B_F^{\varepsilon} : Y_N \rightarrow Y_N$ leaves the
  Gaussian measure ${\mu}^{\varepsilon}$ invariant. More precisely, if
  $\dD$ denotes the gradient with respect to the Fourier monomials
  $(e_k)_{0 < | k | \leqslant N}$ on $Y_N$, then
  \[ \int_{Y_N} (B_F^{\varepsilon} (u) \cdot \dD \Phi (u)) \Psi (u)
     {\mu}^{\varepsilon} (\mathd u) = - \int_{Y_N} \Phi (u)
     B_F^{\varepsilon} (u) \cdot \dD \Psi (u) {\mu}^{\varepsilon}
     (\mathd u) \]
  for all $\Phi, \Psi \in L^2 ({\mu}^{\varepsilon})$ with
  $B_F^{\varepsilon} \cdot \dD \Phi, B_F^{\varepsilon} \cdot \dD \Psi
  \in L^2 ({\mu}^{\varepsilon})$.
\end{lemma}

\begin{proof}
  In this proof it is more convenient to work with the orthonormal basis
  \[ \left\{ \frac{1}{\sqrt{\pi}} \sin (k \cdummy), \frac{1}{\sqrt{\pi}} \cos
     (k \cdummy), 0 < k \leqslant N \right\} \]
  of $Y_N$, rather than with Fourier monomials. We write $(\varphi_k)_{k = 1,
  \ldots, 2 N}$ for an enumeration of these trigonometric functions. Then
  $B^{\varepsilon}_F \cdot \dD$ can also be expressed in terms of the
  $(\varphi_k)$, and we have
  \[ \Phi (u) = f (\langle u, \varphi_1 \rangle, \ldots, \langle u, \varphi_{2
     N} \rangle), \hspace{2em} \Psi (u) = g (\langle u, \varphi_1 \rangle,
     \ldots, \langle u, \varphi_{2 N} \rangle) \]
  for some $f, g : \mathbb{R}^{2 N} \rightarrow \mathbb{R}$. We assume that
  $f$ and $g$ are continuously differentiable, with polynomial growth of the
  first order derivatives. The general case then follows by an approximation
  argument (note that Hermite polynomials of normed linear combinations of $(\langle u, \varphi_k
  \rangle)_k$ form an orthogonal basis of $L^2 ({\mu}^{\varepsilon})$).
  Identifying $Y_N$ with $\mathbb{R}^{2 N}$, we can write
  ${\mu}^{\varepsilon} (\mathd u) = \gamma_{2 N} (u) \mathd u$, where
  $\gamma_{2 N}$ is the density of a $2 N$-dimensional standard normal
  variable. Integrating by parts we therefore have
  \begin{align}\label{eq:antisym pr} \nonumber
     &\int_{Y_N} (B_F^{\varepsilon} (u) \cdot \dD \Phi (u)) \Psi (u)  {\mu}^{\varepsilon} (\mathd u) = - \int_{Y_N} (B_F^{\varepsilon} (u)
     \cdot \dD \Psi (u)) \Phi (u) {\mu}^{\varepsilon} (\mathd u) \\
    &\hspace{70pt} - \int_{Y_N} \sum_{k = 1}^{2 N} (\langle
    \partial_{\langle u, \varphi_k \rangle} B_F^{\varepsilon} (u), \varphi_k
    \rangle - \langle B_F^{\varepsilon} (u), \varphi_k \rangle \langle u,
    \varphi_k \rangle) \Psi (u) \Phi (u) {\mu}^{\varepsilon} (\mathd u)
  \end{align}
  and it suffices to show that the zero order differential operator terms on
  the right hand side vanish. For the first one of them we have
  \begin{align*}
     \sum_{k = 1}^{2 N} \langle \partial_{\langle u, \varphi_k \rangle}
     B_F^{\varepsilon} (u), \varphi_k \rangle & = \sum_{k = 1}^{2 N} \langle
     \partial_{\langle u, \varphi_k \rangle} \varepsilon^{- 1} \partial_x
     \Pi^N_0 \tilde{F} (\varepsilon^{1 / 2} u), \varphi_k \rangle \\
     & = \sum_{k = 1}^{2 N} \langle \varepsilon^{- 1 / 2} \partial_x (\Pi^N_0
     \tilde{F} (\varepsilon^{1 / 2} u) \varphi_k), \varphi_k \rangle \\
     & = - \sum_{k = 1}^{2 N} \langle \varepsilon^{- 1 / 2} \Pi^N_0 \tilde{F}
     (\varepsilon^{1 / 2} u) \varphi_k, \partial_x \varphi_k \rangle \\
     & = - \frac{\varepsilon^{- 1 / 2}}{2} \langle \Pi^N_0 \tilde{F}
     (\varepsilon^{1 / 2} u), \partial_x \sum_{k = 1}^{2 N} \varphi_k^2
     \rangle,
   \end{align*}
  and since $\sin (m x)^2 + \cos (m x)^2 = 1$ the sum of the squares of the
  $\varphi_k$ does not depend on $x$ so its derivative is 0. For the
  remaining term in~(\ref{eq:antisym pr}) we get
  ${\mu}^{\varepsilon}$-almost surely
  \begin{align*}
     \sum_{k = 1}^{2 N} \langle B_F^{\varepsilon} (u), \varphi_k \rangle
     \langle u, \varphi_k \rangle & = \langle B_F^{\varepsilon} (u), u \rangle =
     \langle \varepsilon^{- 1} \partial_x \Pi^N_0 \tilde{F} (\varepsilon^{1 /
     2} u), u \rangle \\
     & = \varepsilon^{- 1} \langle \partial_x \tilde{F} (\varepsilon^{1 / 2} u),
     \Pi^N_0 u \rangle = - \varepsilon^{- 1} \langle \tilde{F} (\varepsilon^{1
     / 2} u), \partial_x \Pi^N_0 u \rangle.
  \end{align*}
  Now observe that there exists $G$ with $G' = \tilde{F}$, and that under
  ${\mu}^{\varepsilon}$ we have $u = \Pi^N_0 u$ almost surely, which
  yields
  \[ - \varepsilon^{- 1} \langle \tilde{F} (\varepsilon^{1 / 2} u), \partial_x
     \Pi^N_0 u \rangle = - \varepsilon^{- 1} \langle G' (\varepsilon^{1 / 2}
     \Pi_0^N u), \partial_x \Pi^N_0 u \rangle = - \varepsilon^{- 3 / 2}
     \langle \partial_x G (\varepsilon \Pi_0^N u), 1 \rangle = 0, \]
  and therefore the proof is complete.
\end{proof}

The previous lemma, together with the reversibility of the Ornstein-Uhlenbeck
dynamics under ${\mu}^{\varepsilon}$, implies that the It{\^o}
SDE~(\ref{eq:HQ-transformed}) has ${\mu}^{\varepsilon}$ as invariant
measure and that for $T > 0$ the time reversed process
$\hat{u}^{\varepsilon}_t = \hat{u}^{\varepsilon}_{T - t}$ solves
\begin{equation}
  \partial_t \hat{u}^{\varepsilon} = \Delta \hat{u}^{\varepsilon} -
  \varepsilon^{- 1} \partial_x \tilde{F} (\varepsilon^{1 / 2} \Pi^N_0
  \hat{u}^{\varepsilon}) + \partial_x \Pi_0^N \hat{\xi}
  \label{eq:full-F-reversed}
\end{equation}
with a time-reversed space-time white noise $\hat{\xi}$.

\section{Boltzmann-Gibbs principle}

In the theory of interacting particle systems the phenomenon that local
quantities of the microscopic fields can be replaced in time averages by
simple functionals of the conserved quantities is called the Boltzmann--Gibbs
principle. In this section we investigate a similar phenomenon in order to
control the antisymmetric drift term
\begin{equation}
  \label{eq:HQ-drift} \int_0^t \varepsilon^{- 1} \partial_x \tilde{F}
  (\varepsilon^{1 / 2} u_s^{\varepsilon} (x)) \mathd s
\end{equation}
as $N \rightarrow + \infty$. Note that since $\varepsilon = 1 / 2N$ and
$u^{\varepsilon} = \Pi_0^N u^{\varepsilon}$ we have $\mathbb{E}
[(\varepsilon^{1 / 2} u_s^{\varepsilon} (x))^2] = 1$ for all $N$, and
therefore the Gaussian random variables $(\varepsilon^{1 / 2}
u^{\varepsilon}_s (x))_N$ stay bounded in $L^2$ for fixed $(s, x)$, but for
large $N$ there will be wild fluctuations in $(s, x)$. We show that
the quantity in~(\ref{eq:HQ-drift}) can be replaced by simpler expressions
that are constant, linear, or quadratic in $u^{\varepsilon}$.

\subsection{A first computation}

In the following we use $\eta$ to denote a generic space white noise and we
write ${\mu}$ for its law, and $G \in C (\mathbb{R}, \mathbb{R})$
denotes a generic continuous function. A first interesting computation is to
consider the random field $x \mapsto G (\varepsilon^{1 / 2} \Pi^N_0 \eta (x))$
and to derive its chaos expansion in the variables $(\eta_k)_k$ where $\eta_k
= \langle \eta, e_{- k} \rangle$ are the Fourier coordinates of $\eta$. To do
so consider the standard (recall that $\varepsilon = (2 N)^{- 1}$) Gaussian
random variable
\[ \eta^N (x) = \varepsilon^{1 / 2} \Pi^N_0 \eta (x) = \varepsilon^{1 / 2}
   \sum_{0 < | k | \leqslant N} e_k (x) \eta_k, \]
and observe that the chaos expansion in $L^2 (\tmop{law} (\eta^N (x)))$ yields
\[ G (\eta^N (x)) = \sum_{n \geqslant 0} c_n (G) H_n (\eta^N (x)), \]
where $H_n$ is the $n$-th Hermite polynomial and
\[ c_n (G) = \frac{1}{n!} \mathbb{E} [G (\eta^N (x)) H_n (\eta^N (x))] =
   \frac{1}{n!} \int_{\mathbb{R}} G (x) H_n (x) \gamma (x) \mathd x, \]
where $\gamma$ is the standard Gaussian density. Since $H_n (x) = (- 1)^n
e^{x^2 / 2} \partial^n_x e^{- x^2 / 2}$, we get
\[ c_n (G) = \frac{1}{n!} \int_{\mathbb{R}} G (x) H_n (x) (- 1)^n
   \partial^n_x \gamma (x) \mathd x = \frac{\psi_G^{(n)} (0)}{n!}, \]
where $\psi_G (\lambda) =\mathbb{E} [G (\lambda + \eta^N (x))]$.

Our next aim is to relate the Hermite polynomials of $\eta^N (x)$ with the
Wick powers of the family $(\eta_k)_k$. To do so we observe that the monomials
$H_n (\eta^N (x))$ are the coefficients of the powers of $\lambda$ in $\exp
(\lambda \eta^N (x) - \lambda^2 / 2)$, and on the other side
\[ \sum_n \frac{\lambda^n}{n!} H_n (\eta^N (x)) = \exp (\lambda \eta^N (x) -
   \lambda^2 / 2) = \exp \Big( \lambda \varepsilon^{1 / 2} \sum_{0 < | k |
   \leqslant N} e_k (x) \eta_k - \frac{1}{2} \sum_{0< | k | \leqslant N} (\lambda
   \varepsilon^{1 / 2})^2 \Big) . \]
Writing $\llbracket \cdummy \rrbracket_n$ for the projection onto the $n$-th
homogeneous chaos generated by $\eta$, we have
\[ \exp \Big( \sum_{0 < | k | \leqslant N} {\mu}_k \eta_k - \frac{1}{2}
   \sum_{0 < | k | \leqslant N} {\mu}_k {\mu}_{- k} \Big) = \sum_{k_1
   \cdots k_n} \frac{{\mu}_{k_1} \cdots {\mu}_{k_n}}{n!} \llbracket
   \eta_{k_1} \cdots \eta_{k_n} \rrbracket_n, \]
where the sum on the right hand side and all the following sums in $k_1 \dots k_n$ are over $0 < | k_1 |, \ldots, | k_n |
\leqslant N$. Setting ${\mu}_k = \varepsilon^{1 / 2} \lambda e_k(x)$
and identifying the coefficients for different powers of $\lambda$, we get
\[ H_n (\varepsilon^{1/2} \Pi_0^N \eta (x)) = \varepsilon^{n/2} \sum_{k_1 \cdots k_n} \frac{e^{i (k_1 + \cdots +
   k_n) x}}{(2 \pi)^{n / 2}} \llbracket \eta_{k_1} \cdots \eta_{k_n}
   \rrbracket_n, \]
which can also be obtained by writing $H_n (\varepsilon^{1/2} \Pi_0^N \eta (x)) = \llbracket
(\varepsilon^{1/2} \Pi_0^N \eta (x))^n \rrbracket_n$ and expanding the power $(\cdummy)^n$
inside the projection. We can thus represent the function $G (\eta^N (x))$ as
\[ G (\eta^N (x)) = \sum_{n \geqslant 0}^n c_n (G)  H_n
   (\varepsilon^{n / 2}\Pi_0^N \eta (x)) = \sum_{n \geqslant 0} c_n (G) \varepsilon^{n / 2}
   \sum_{k_1, \ldots, k_n} \frac{e^{i (k_1 + \cdots + k_n) x}}{(2 \pi)^{n /
   2}} \llbracket \eta_{k_1} \cdots \eta_{k_n} \rrbracket_n . \]
If $\varphi \in C^{\infty} (\mathbb{T})$ is a test function, we get
\begin{equation}
  \label{eq:G-chaos-exp} \langle G (\eta^N), \varphi \rangle = \sum_{n
  \geqslant 0} c_n (G) \varepsilon^{n / 2} \sum_{k_1, \ldots, k_n}
  \frac{\hat{\varphi} (- k_1 - \cdots - k_n)}{(2 \pi)^{(n - 1) / 2}}
  \llbracket \eta_{k_1} \cdots \eta_{k_n} \rrbracket_n .
\end{equation}
So in particular the $q$-th Littlewood-Paley block of $G (\eta^N)$ is given by
\[ \Delta_q G (\eta^N) (x) = \sum_{n \geqslant 0} c_n (G) \varepsilon^{n / 2}
   \sum_{k_1, \ldots, k_n} \theta_q (k_1 + \cdots + k_n) \frac{e^{i (k_1 + \cdots + k_n) x}}{(2 \pi)^{n / 2}} \llbracket
   \eta_{k_1} \cdots \eta_{k_n} \rrbracket_n, \]
where $(\theta_q)_{q \geqslant -1}$ is a dyadic partition of unity, and
\begin{align*}
   \mathbb{E} [|\Delta_q (G (\eta^N) - c_0(G))(x)|^2] & \leqslant \sum_{n
   \geqslant 1} c_n (G)^2 z_n \varepsilon^n \sum_{k_1, \ldots, k_n} \theta_q
   (k_1 + \cdots + k_n)^2 \\
   & \lesssim \sum_{n \geqslant 1} c_n (G)^2 z_n \varepsilon^n N^{n - 1} (2^q
   \wedge N) \lesssim \varepsilon \sum_{n \geqslant 1} c_n (G)^2 z_n (2^q
   \wedge N),
\end{align*}
where $z_n =\max_{k_1\dots k_n}\mathbb{E} [| \llbracket \eta_{k_1} \cdots \eta_{k_n}
\rrbracket_n / (2 \pi)^{n} |^2] \leqslant n!$ is a combinatorial
factor. We thus obtain
\[ \mathbb{E} [\| \Delta_q (G (\eta^N) - \psi_G (0)) \|_{L^2
   (\mathbb{T})}^2] \lesssim \min \{ \varepsilon^{p / 2} 2^{qp / 2}, 1 \}, \]
uniformly in $N$, and then
\begin{align*}
   \mathbb{E} \left[ \left| \int_s^t \Delta_q (G (\varepsilon^{1 / 2}
   u_r^{\varepsilon} (x)) - \psi_G (0)) \mathd r \right|^2 \right] & \leqslant |
   t - s | \int_s^t \mathbb{E} [| \Delta_q (G (\varepsilon^{1 / 2}
   u_r^{\varepsilon} (x)) - \psi_G (0)) |^2] \mathd r \\
   & \lesssim | t - s |^2 \min \{ \varepsilon^{p / 2} 2^{qp / 2}, 1 \},
\end{align*}
where in the last step we used that $\varepsilon^{1 / 2} u^{\varepsilon}_r$
has the same distribution as $\eta^N$, which easily implies the following
result.

\begin{lemma}
  Assume that $\mathbb{E} [| G (U) |^2] < \infty$ for a  standard
  normal variable $U$, and let $c_0 (G) =\mathbb{E} [G (U)]$. Then
  \[ \lim_{N \rightarrow \infty} \int_0^t G (\varepsilon^{1 / 2}
     u_s^{\varepsilon} (x)) \mathd s = c_0 (G) t, \]
  where the convergence is in $C ([0, T], H^{0 -})$. If $c_0 (G) = 0$, then
  \[ \varepsilon^{- 1 / 2} \int_0^t G (\varepsilon^{1 / 2} u_s^{\varepsilon}
     (x)) \mathd s \]
  is bounded in $C ([0, T], H^{- 1 / 2 -})$.
\end{lemma}

To analyse the for us interesting case with $c_0 (G) = 0$ we need a more
refined argument which is provided by the regularization by noise of
controlled paths.

\subsection{Regularization by noise}

Let us write $\cL^{\varepsilon}_0$ for the generator of the mollified
Ornstein--Uhlenbeck process
\[ \partial_t X^{\varepsilon} = \Delta X^{\varepsilon} + \partial_x \Pi_0^N
   \xi . \]
The basic tool which allows us to control time integrals such as $\int_0^t G
(\varepsilon^{1 / 2} u_s^{\varepsilon} (x)) \mathd s$ is given by the It{\^o}
trick. To state it, we define for $\Psi \in L^2 ({\mu}^{\varepsilon})$
\[ \mathcal{E}^{\varepsilon} (\Psi) \assign \sum_{0 < | k | \leqslant N} k^2 |
   \mathD_k \Psi |^2, \]
where $\mathD_k$ is the directional derivative in $e_k$.

\begin{lemma}[It{\^o} trick]
  
  For $\Psi \in \mathrm{\tmop{dom}} \left( \cL^{\varepsilon}_0 \right)$ and $T
  > 0$, $p \geqslant 1$ we have
  \[ \mathbb{E} \left[ \sup_{t \in [0, T]} \left| \int_0^t
     \cL^{\varepsilon}_0 \Psi (u^{\varepsilon}_s) \mathd s \right|^p \right]
     \lesssim T^{p / 2} \mathbb{E} [\mathcal{E}^{\varepsilon} (\Psi)^{p / 2}]
     . \]
\end{lemma}

The proof is given in \cite{GJ13, GP15} and extends without difficulty
to our setting, so we do not repeat the arguments here.

To apply the It{\^o} trick we need to solve the Poisson equation. In our
setting this can be done efficiently by using the chaos expansion~(\ref{eq:G-chaos-exp}). Recall that we wrote $\eta_k = \langle
\eta, e_k \rangle$ for the Fourier coefficients of a truncated spatial white
noise $\Pi_0^N \eta$ (which therefore has law ${\mu}^{\varepsilon}$), and
that $\llbracket \cdummy \rrbracket_n$ denotes the projection onto the $n$-th
chaos. We need to compute $\cL^{\varepsilon}_0 \llbracket \eta_{k_1} \ldots
\eta_{k_n} \rrbracket_n$, as these are the random variables appearing in a
general chaos expansion. Let us start by considering $\varphi \in Y_N =
\Pi_0^N L^2 (\mathbb{T}, \mathbb{R})$ with $\| \varphi \|_{L^2} = 1$ for
which we have $\llbracket \langle \eta, \varphi \rangle^n \rrbracket_n = H_n
(\langle \eta, \varphi \rangle)$, where $H_n$ is the $n$-th Hermite polynomial.
It{\^o}'s formula gives
\[ \mathd H_n (\langle X^{\varepsilon}_t, \varphi \rangle) = H_n' (\langle
   X^{\varepsilon}_t, \varphi \rangle) \langle X^{\varepsilon}_t, \Delta
   \varphi \rangle \mathd t + H_n'' (\langle X^{\varepsilon}_t, \varphi
   \rangle) \langle \Pi_0^N \partial_x \varphi, \Pi_0^N \partial_x \varphi \rangle
   \mathd t + \mathd M_t, \]
with a square integrable martingale $M$. The Hermite polynomials satisfy $H_n'
= n H_{n - 1}$, so we get
\begin{align*}
   & H_n' (\langle X^{\varepsilon}_t, \varphi \rangle) \langle X^{\varepsilon}_t, \Delta \varphi \rangle \mathd t + H_n'' (\langle X^{\varepsilon}_t, \varphi \rangle) \langle \Pi_0^N \partial_x \varphi, \Pi_0^N \partial_x \varphi \rangle \\
   &\hspace{50pt} = n H_{n - 1} (\langle X^{\varepsilon}_t, \varphi \rangle) H_1 (\langle  X^{\varepsilon}_t, \Delta \varphi \rangle) - n (n - 1) H_{n - 2} (\langle
   X^{\varepsilon}_t, \varphi \rangle) \langle \Pi_0^N \varphi, \Pi_0^N \Delta
   \varphi \rangle .
\end{align*}
The projection onto the $n$-th chaos of the first term is explicitly given by
\begin{align*}
   \llbracket H_{n - 1} (\langle X^{\varepsilon}_t, \varphi \rangle) H_1
   (\langle X^{\varepsilon}_t, \Delta \varphi \rangle) \rrbracket_n & =
   \llbracket \llbracket \langle X^{\varepsilon}_t, \varphi \rangle^{n - 1}
   \rrbracket_{n - 1} \llbracket \langle X^{\varepsilon}_t, \Delta \varphi
   \rangle \rrbracket_1 \rrbracket_n \\
   & = \llbracket \langle X^{\varepsilon}_t, \varphi \rangle^{n - 1}
   \rrbracket_{n - 1} \llbracket \langle X^{\varepsilon}_t, \Delta \varphi
   \rangle \rrbracket_1\\
   & \qquad  - (n - 1) \llbracket \langle X^{\varepsilon}_t,
   \varphi \rangle^{n - 2} \rrbracket_{n - 2} \langle \Pi_0^N \varphi, \Pi_0^N
   \Delta \varphi \rangle,
\end{align*}
which is obtained by contracting $\langle X^{\varepsilon}_t, \Delta \varphi
\rangle$ with each of the $n - 1$ variables $\langle X^{\varepsilon}_t,
\varphi \rangle$ inside the projector $\llbracket \cdummy \rrbracket_{n - 1}$.
Therefore, we have
\begin{align*} \mathd H_n (\langle X^{\varepsilon}_t, \varphi \rangle) & = n \llbracket H_{n
   - 1} (\langle X^{\varepsilon}_t, \varphi \rangle) H_1 (\langle
   X^{\varepsilon}_t, \Delta \varphi \rangle) \rrbracket_n \mathd t + \mathd
   M_t \\
   & = n \llbracket \langle X^{\varepsilon}_t, \varphi \rangle^{n - 1} \langle
   X^{\varepsilon}_t, \Delta \varphi \rangle \rrbracket_n \mathd t + \mathd
   M_t,
\end{align*}
which shows that
\[ \cL^{\varepsilon}_0 \llbracket \langle \eta, \varphi \rangle^n \rrbracket_n
   = n \llbracket \langle \eta, \varphi \rangle^{n - 1} \langle \eta, \Delta
   \varphi \rangle \rrbracket_n . \]
So far we assumed $\| \varphi \|_{L^2} = 1$, but actually this last formula is
invariant under scaling so it extends to all $\varphi \in \Pi_0^N
L^2 (\mathbb{T}, \mathbb{R})$, and then to $\varphi \in \Pi_0^N L^2
(\mathbb{T}, \mathbb{C})$, and for general products we obtain by
polarization
\[ \cL^{\varepsilon}_0 \llbracket \langle \eta, \varphi_1 \rangle \ldots
   \langle \eta, \varphi_n \rangle \rrbracket_n = \sum_{k = 1}^n
   \left\llbracket \langle \eta, \varphi_1 \rangle \ldots \cancel{\langle \eta,
   \varphi_k \rangle} \ldots \langle \eta, \varphi_n \rangle \langle \eta,
   \Delta \varphi_k \rangle \right\rrbracket_n . \]
So finally we deduce that
\begin{equation}
  \cL^{\varepsilon}_0 \llbracket \eta_{k_1} \cdots \eta_{k_n} \rrbracket = -
  (k_1^2 + \cdots + k_n^2) \llbracket \eta_{k_1} \cdots \eta_{k_n} \rrbracket
\end{equation}
for all $0 < | k_1 |, \ldots, | k_n | \leqslant N$. Combining that formula
with~(\ref{eq:G-chaos-exp}), we obtain the following lemma.

\begin{lemma}
  \label{lem:Poisson-eq}Consider a function of the form $\Phi (\eta) = \langle
  G (\varepsilon^{1 / 2} \Pi_0^N \eta), \varphi \rangle$ and assume that
  $\mathbb{E} [G (U)] = 0$, where $U$ is a standard normal variable, or that
  $\hat{\varphi} (0) = 0$. Then the solution $\Psi$ to the Poisson equation
  $\cL_0^{\varepsilon} \Psi = \Phi$ is explicitly given by
  \[ \Psi (\eta) = - \sum_{n \geqslant 1} c_n (G) \varepsilon^{n / 2}
     \sum_{k_1 \cdots k_n} \frac{\hat{\varphi} (- k_1 - \cdots - k_n)}{(2
     \pi)^{(n - 1) / 2}} \frac{\llbracket \eta_{k_1} \cdots \eta_{k_n}
     \rrbracket_n}{(k_1^2 + \cdots + k_n^2)}, \]
  where the sum is over all $0 < | k_1 |, \ldots, | k_n | \leqslant N$.
\end{lemma}

\begin{remark}
  Incidentally note that the solution can be represented as
  \begin{align*}
     \Psi (\eta) & = - \int_0^{\infty} \mathd t \sum_{n \geqslant 1} c_n (G)
     \varepsilon^{n / 2} \sum_{k_1 \cdots k_n} e^{- (k_1^2 + \cdots + k_n^2)
     t} \frac{e^{i (k_1 + \cdots + k_n) x}}{(2 \pi)^{n / 2}} \llbracket
     \eta_{k_1} \cdots \eta_{k_n} \rrbracket_n \\
     & = - \int_0^{\infty} \mathd t G (\varepsilon^{1 / 2} (e^{\Delta t} \Pi_0^N
     \eta) (x)) .
   \end{align*}
\end{remark}

To apply the It{\^o} trick we need to compute $\mathcal{E} (\Psi) = \sum_k k^2
\mathD_{- k} \Psi \mathD_k \Psi$ for the solution $\Psi$ of the Poisson
equation. For that purpose consider again $\varphi \in Y_N$ with $\| \varphi
\|_{L^2} = 1$ and $H_n (\langle \eta, \varphi \rangle) = \llbracket \langle
\eta, \varphi \rangle^n \rrbracket_n$, for which we have
\[ \mathD_k H_n (\langle \eta, \varphi \rangle) = H'_n (\langle \eta, \varphi
   \rangle) \langle e_k, \varphi \rangle = n H_{n - 1} (\langle \eta, \varphi
   \rangle) \langle e_k, \varphi \rangle = n \llbracket \langle \eta, \varphi
   \rangle^{n - 1} \rrbracket_{n - 1} \langle e_k, \varphi \rangle , \]
so by polarization
\begin{equation}
  \label{eq:Dk-on-Wick} \mathD_k \llbracket \eta_{k_1} \cdots \eta_{k_n}
  \rrbracket_n = \sum_j \mathbf{1}_{k_j = k} \llbracket \eta_{k_1} \cdots
  \cancel{\eta_{k_j}} \cdots \eta_{k_n} \rrbracket_{n - 1} .
\end{equation}
To prove the Boltzmann--Gibbs principle we need one more auxiliary
result.

\begin{lemma}
  \label{lem:BG-prep}For all $M \leqslant N$, $\ell \in \mathbb{Z}$ and $0
  \leqslant s < t < \infty$ we have the estimate
  \[ \mathbb{E} \left[ \left| \int_s^t \langle \partial_x (\Pi_0^M
     u^{\varepsilon}_r)^2, e_{- \ell} \rangle \mathd r \right|^2 \right]
     \lesssim \ell^2 | t - s |^2 M. \]
\end{lemma}

\begin{proof}
  We simply bound
  \[ \mathbb{E} \left[ \left| \int_s^t \langle \partial_x (\Pi_0^M
     u^{\varepsilon}_r)^2, e_{- \ell} \rangle \mathd r \right|^2 \right]
     \leqslant | t - s | \int_s^t \mathbb{E} [| \langle \partial_x (\Pi_0^M
     u^{\varepsilon}_r)^2, e_{- \ell} \rangle |^2] \mathd r, \]
  and since we can replace $(\Pi_0^M u^{\varepsilon}_r)^2$ by $(\Pi_0^M
  u^{\varepsilon}_r)^2 -\mathbb{E} [(\Pi_0^M u^{\varepsilon}_r)^2]$, the
  integrand is given by
  \begin{align*}
     \mathbb{E} [| \langle \partial_x (\Pi_0^M u^{\varepsilon}_r)^2, e_{-
     \ell} \rangle |^2] & = \ell^2 \int_{\mathbb{T}} \mathd x
     \int_{\mathbb{T}} \mathd x' \mathbb{E} [\llbracket (\Pi_0^M
     u^{\varepsilon}_r (x))^2 \rrbracket_2 \llbracket (\Pi_0^M
     u^{\varepsilon}_r (x'))^2 \rrbracket_2] \\
     & \lesssim \ell^2 \int_{\mathbb{T}} \mathd x \int_{\mathbb{T}} \mathd x'
     | \mathbb{E} [\Pi_0^M u^{\varepsilon}_r (x) \Pi_0^M u^{\varepsilon}_r
     (x')] |^2 .
  \end{align*}
  The expectation on the right hand side can be explicitly computed as
  \begin{align*}
     | \mathbb{E} [\Pi_0^M u^{\varepsilon}_r (x) \Pi_0^M u^{\varepsilon}_r
     (x')] | & = \Big| \sum_{0 < | k | \leqslant M} e^{i k (x - x')} \Big| =
     \Big| \frac{\cos (M (x - x')) - \cos ((M + 1) (x - x'))}{1 - \cos (x -
     x')} - 1 \Big| \\
     & \lesssim \min \{ M, | x - x' |^{- 1} \},
  \end{align*}
  for which
  \[ \int_{\mathbb{T}} \mathd x \int_{\mathbb{T}} \mathd x' \min \{ M, | x -
     x' |^{- 1} \}^2 \mathd x \lesssim M, \]
  and therefore the claim follows.
\end{proof}

\begin{proposition}[Boltzmann--Gibbs principle]
  \label{prop:BG}
  
  Let $G, G' \in L^2 (\nu)$, where $\nu$ denotes the law of a standard normal
  variable. Then for all $\ell \in \mathbb{Z}$ and $0 \leqslant s < t
  \leqslant s + 1$ and all $\kappa > 0$
  \[ \mathbb{E} \left[ \left| \int_s^t \langle \varepsilon^{- 1} \partial_x
     \Pi_0^N G (\varepsilon^{1 / 2} u^{\varepsilon}_r) - \varepsilon^{- 1 /
     2} c_1 (G) \partial_x \Pi_0^N u^{\varepsilon}_r, e_{- \ell} \rangle
     \mathd r \right|^2 \right] \lesssim | t - s |^{3 / 2 - \kappa} \ell^2
     \int_{\mathbb{R}} | G' (x) |^2 \nu (\mathd x) \]
  uniformly in $N \in \mathbb{N}$, and for all $M \leqslant N$
  \begin{gather*}
     \mathbb{E} \left[ \left| \int_s^t \langle \varepsilon^{- 1} \partial_x
     \Pi_0^N G (\varepsilon^{1 / 2} u^{\varepsilon}_r) - \varepsilon^{- 1 /
     2} c_1 (G) \partial_x \Pi_0^N u^{\varepsilon}_r - c_2 (G) \partial_x
     (\Pi_0^M u^{\varepsilon}_r)^2, e_{- \ell} \rangle \mathd r \right|^2
     \right] \\
     \lesssim | t - s | \ell^2 (M^{- 1} + \varepsilon \log^2 N)
     \int_{\mathbb{R}} | G' (x) |^2 \nu (\mathd x) . 
   \end{gather*}
\end{proposition}

\begin{proof}
  We first show the second bound. Towards this end note that by
  Lemma~\ref{lem:Poisson-eq} the solution $\Psi$ to
  \[ \cL^{\varepsilon}_0 \Psi (\eta) = - \varepsilon^{- 1} \langle G
     (\varepsilon^{1 / 2} \Pi_0^N \eta) - c_1 (G) \varepsilon^{1 / 2} \Pi_0^N
     \eta - c_2 (G) (\varepsilon^{1 / 2} \Pi_0^M \eta)^2, \partial_x \Pi_0^N
     e_{- \ell} \rangle \]
  is given by
  \begin{align*}
    \Psi (\eta) & = c_2 (G) \sum_{k_1, k_2} \mathbf{1}_{| k_1 | \vee | k_2 |
     \geqslant M} \mathbf{1}_{0 < | \ell | \leqslant N} (i \ell)
     \frac{\mathbf{1}_{k + k_1 = \ell}}{(2\pi)^{1/2}} \frac{\llbracket \eta_{k_1} \eta_{k_2}
     \rrbracket_2}{(k_1^2 + k_2^2)} \\
     & \qquad + \sum_{n \geqslant 3} c_n (G) \varepsilon^{n / 2 - 1} \sum_{k_1 \cdots
     k_n} \mathbf{1}_{0 < | \ell | \leqslant N} (i \ell) \frac{\mathbf{1}_{k_1 +
     \cdots + k_n = \ell}}{(2\pi)^{(n-1)/2}} \frac{\llbracket \eta_{k_1} \cdots \eta_{k_n}
     \rrbracket_n}{(k_1^2 + \cdots + k_n^2)},
  \end{align*}
  where it is understood that all sums sums in $k_i$ are over $0<|k_i|\leqslant N$. Therefore~(\ref{eq:Dk-on-Wick}) yields for $0 < | \ell | \leqslant N$
  \begin{align*}
     \mathD_k \Psi (\eta) & = c_2 (G) 2 \sum_{k_1} \mathbf{1}_{| k | \vee | k_1
     | \geqslant M} i \ell \frac{\mathbf{1}_{k + k_1 = \ell}}{(2\pi)^{1/2}} \frac{\llbracket
     \eta_{k_1} \rrbracket_1}{(k^2 + k_1^2)} \\
     & \qquad  - \sum_{n \geqslant 1} c_{n + 1} (G) \varepsilon^{n / 2 - 1}  (n + 1)
     \sum_{k_1 \cdots k_n} i \ell \frac{\mathbf{1}_{k + k_1 + \cdots + k_n = \ell}}{(2\pi)^{n/2}}
     \frac{\llbracket \eta_{k_1} \cdots \eta_{k_n} \rrbracket_n}{(k^2 + k_1^2
     + \cdots + k_n^2)} .
  \end{align*}
  Applying the It{\^o} trick we then get
  \begin{align*}
     &\mathbb{E} \left[ \left| \int_s^t \langle \varepsilon^{- 1} \partial_x
     \Pi_0^N G (\varepsilon^{1 / 2} u^{\varepsilon}_r) - \varepsilon^{- 1 /
     2} c_1 (G) \partial_x \Pi_0^N u^{\varepsilon}_r - c_2 (G) \partial_x
     (\Pi_0^M u^{\varepsilon}_r)^2, e_{- \ell} \rangle \mathd r \right|^2
     \right] \\
     &\hspace{40pt} \lesssim | t - s | \sum_{0 < | k | \leqslant N} k^2 \mathbb{E} [|
     \mathD_k \Psi |^2] \\
     &\hspace{40pt} = | t - s | \sum_{0 < | k | \leqslant N} k^2 c_2 (G) ^2 2^2 \ell^2
     \sum_{k_1} \mathbf{1}_{| k | \vee | k_1 | \geqslant M} \frac{\mathbf{1}_{k + k_1 = \ell}}{2\pi} \frac{\mathbb{E} [| \llbracket \eta_{k_1} \rrbracket_1
     |^2]}{(k^2 + k_1^2)^2} \\
     &\hspace{40pt}\qquad + | t - s | \sum_{0 < | k | \leqslant N} k^2 \sum_{n \geqslant 2} c_{n +
     1} (G) ^2 \varepsilon^{(n + 1) - 2} (n + 1)^2 \ell^2  \\
     &\hspace{160pt} \times \sum_{k_1 \cdots
     k_n} \frac{\mathbf{1}_{k + k_1 + \cdots + k_n = \ell}}{(2\pi)^{n}} \frac{\mathbb{E} [|
     \llbracket \eta_{k_1} \cdots \eta_{k_n} \rrbracket_n |^2]}{(k^2 + k_1^2 +
     \cdots + k_n^2)^2} \\
     &\hspace{40pt} = | t - s | \sum_{n \geqslant 1} A_n,
  \end{align*}
  where the $(A_n)$ are implicitly defined by the equation. Now $\mathbb{E}
  [| \llbracket \eta_{k_1} \cdots \eta_{k_n} \rrbracket_n |^2] \leqslant n!$ for all
  $k_1, \ldots, k_n$, so that
  \begin{align*}
     A_1 & \lesssim \sum_{0 < | k |, | k_1 | \leqslant N} k^2 c_2 (G) ^2 \ell^2
     \mathbf{1}_{k + k_1 = \ell} \frac{\mathbf{1}_{| k | \vee | k_1 |
     \geqslant M}}{(k^2 + k_1^2)^2} \leqslant \sum_{0 < | k |, | k_1 |
     \leqslant N} c_2 (G) ^2 \ell^2 \mathbf{1}_{k + k_1 = \ell}
     \frac{\mathbf{1}_{| k | \vee | k_1 | \geqslant M}}{k^2 + k_1^2} \\
     & \lesssim c_2 (G) ^2 \ell^2 \sum_{0 < | k | < \infty}
     \frac{\mathbf{1}_{\ell \neq k} \mathbf{1}_{| k | \vee | \ell - k |
     \geqslant M}}{k^2 + (\ell - k)^2} \leqslant c_2 (G) ^2 \ell^2 \sum_{0 < |
     k | < \infty} \left( \frac{\mathbf{1}_{\ell \neq k}}{M^2 + (\ell - k)^2}
     + \frac{\mathbf{1}_{\ell \neq k}}{k^2 + M^2} \right) \\
     & \lesssim c_2 (G) ^2 \ell^2 M^{- 1},
  \end{align*}
  while for $n > 1$
  \begin{align*}
     A_n & = \sum_{0 < | k | \leqslant N} k^2 c_{n + 1} (G) ^2 \varepsilon^{n -
     1} (n + 1)^2 \ell^2 \sum_{k_1 \cdots k_n} \frac{\mathbf{1}_{k + k_1 + \cdots + k_n = \ell}}{(2\pi)^{n/2}} \frac{\mathbb{E} [| \llbracket \eta_{k_1} \cdots \eta_{k_n}
     \rrbracket_n |^2]}{(k^2 + k_1^2 + \cdots + k_n^2)^2} \\
    & = \varepsilon^{n - 1} \ell^2 (n + 1)^2 c_{n + 1} (G) ^2 n! \sum_{0 < | k
     |, | k_1 |, \ldots, | k_n | \leqslant N} k^2 \frac{\mathbf{1}_{k + k_1 +
     \cdots + k_n = \ell}}{(k^2 + k_1^2 + \cdots + k_n^2)^2} \\
     & \leqslant \varepsilon^{n - 1} \ell^2 (n + 1)^2 c_{n + 1} (G) ^2 n!
     \sum_{0 < | k_1 |, \ldots, | k_n | \leqslant N} \frac{1}{k_1^2 + \cdots +
     k_n^2} \\
     & \leqslant \varepsilon^{n - 1} \ell^2 (n + 1)^2 c_{n + 1} (G) ^2 n!
     \sum_{0 < | k_1 |, \ldots, | k_n | \leqslant N} \frac{1}{k_1^2 + k_2^2}
  \\
     & = \varepsilon^{n - 1} \ell^2 (n + 1)^2 c_{n + 1} (G) ^2 n!N^{n - 2}
     \sum_{0 < | k_1 |, | k_2 | \leqslant N} \frac{1}{k_1^2 + k_2^2} \lesssim
     \varepsilon \ell^2 (n + 1)^2 c_{n + 1} (G) ^2 n! \log^2 N.
  \end{align*}
  The sum over $n$ is bounded by
  \[ \sum_{n = 2}^{\infty} c_{n + 1} (G) ^2 n! (n + 1)^2 = \sum_{n =
     1}^{\infty} n c_n (G) ^2 n! \lesssim \int_{\mathbb{R}} | G' (x) |^2 \nu
     (\mathd x), \]
  so that overall we get
  \begin{align}\label{eq:BG-pr1} \nonumber
     &\mathbb{E} \left[ \left| \int_s^t \langle \varepsilon^{- 1} \partial_x
     \Pi_0^N G (\varepsilon^{1 / 2} u^{\varepsilon}_r) - \varepsilon^{- 1 /
     2} c_1 (G) \partial_x \Pi_0^N u^{\varepsilon}_r - c_2 (G) \partial_x
     (\Pi_0^M u^{\varepsilon}_r)^2, e_{- \ell} \rangle \mathd r \right|^2
     \right] \\
     &\hspace{80pt} \lesssim | t - s | \ell^2 (M^{- 1} + \varepsilon \log^2
    N) \int_{\mathbb{R}} | G' (x) |^2 \nu (\mathd x),
  \end{align}
  which is our second claimed bound.
  
  To get the first bound, we take $M \simeq | t - s |^{- 1 / 2}$
  in~(\ref{eq:BG-pr1}) (which requires $N > | t - s |^{- 1 / 2}$), and
  combine this with Lemma~\ref{lem:BG-prep} to obtain
  \begin{gather*}
     \mathbb{E} \left[ \left| \int_s^t \langle \varepsilon^{- 1} \partial_x
     \Pi_0^N G (\varepsilon^{1 / 2} u^{\varepsilon}_r) - \varepsilon^{- 1 /
     2} c_1 (G) \partial_x \Pi_0^N u^{\varepsilon}_r, e_{- \ell} \rangle
     \mathd r \right|^2 \right] \\
     \lesssim | t - s | \ell^2 (M^{- 1} + \varepsilon \log^2 N + | t - s | M)
     \int_{\mathbb{R}} | G' (x) |^2 \nu (\mathd x) \lesssim | t - s |^{3 / 2 -
     \kappa} \ell^2 \int_{\mathbb{R}} | G' (x) |^2 \nu (\mathd x) .
  \end{gather*}
  If $N \leqslant | t - s |^{- 1 / 2}$ we use another estimate: ass in the
  proof of Lemma~\ref{lem:BG-prep} we have
  \begin{align*}
     &\mathbb{E} \left[ \left| \int_s^t \langle \varepsilon^{- 1} \partial_x
     \Pi_0^N G (\varepsilon^{1 / 2} u^{\varepsilon}_r) - \varepsilon^{- 1 /
     2} c_1 (G) \partial_x \Pi_0^N u^{\varepsilon}_r, e_{- \ell} \rangle
     \mathd r \right|^2 \right] \\
     &\hspace{60pt} \leqslant | t - s |^2 \mathbb{E} [| \langle \varepsilon^{- 1} \partial_x
     \Pi_0^N G (\varepsilon^{1 / 2} u^{\varepsilon}_0) - \varepsilon^{- 1 /
     2} c_1 (G) \partial_x \Pi_0^N u^{\varepsilon}_0, e_{- \ell} \rangle |^2]
  \\
     &\hspace{60pt} \lesssim | t - s |^2 \sum_{n \geqslant 2} \ell^2 \int_{\mathbb{T}}
     \mathd x \int_{\mathbb{T}} \mathd x' \varepsilon^{- 2} c_n (G)^2
     \mathbb{E} [H_n (\varepsilon^{1 / 2} u^{\varepsilon}_0 (x)) H_n
     (\varepsilon^{1 / 2} u^{\varepsilon}_0 (x'))], \\
     &\hspace{60pt} \lesssim | t - s |^2 \sum_{n \geqslant 2} \ell^2 \int_{\mathbb{T}}
     \mathd x \int_{\mathbb{T}} \mathd x' \varepsilon^{- 2} c_n (G)^2 n! |
     \mathbb{E} [\varepsilon^{1 / 2} u^{\varepsilon}_r (x) \varepsilon^{1 /
     2} u^{\varepsilon}_r (x')]^n | \\
     &\hspace{60pt} \lesssim | t - s |^2 \sum_{n \geqslant 2} \ell^2 \int_{\mathbb{T}}
     \mathd x \int_{\mathbb{T}} \mathd x' \varepsilon^{- 2} c_n (G)^2
     \varepsilon^n n! \min \{ N, | x - x' |^{- 1} \}^n \\
     &\hspace{60pt} \lesssim | t - s |^2 \sum_{n \geqslant 2} \ell^2 \varepsilon^{- 2} c_n
     (G)^2 \varepsilon^n n!N^{n - 1} \lesssim | t - s |^2 \sum_{n \geqslant 2}
     \ell^2 \varepsilon^{- 1} c_n (G)^2 n! \\
     &\hspace{60pt} \lesssim \ell^2 | t - s |^{3 / 2} \int_{\mathbb{R}} | G' (x) |^2 \nu
     (\mathd x),
  \end{align*}
  where in the last step we used that $| t - s |^{- 1 / 2} N^{- 1} \geqslant
  1$.
\end{proof}

\section{The invariance principle}

We now have all the tools to prove the convergence of $(u^{\varepsilon})$ to
an energy solution of the stochastic Burgers equation. We
proceed in two steps. First we establish the tightness of
$(u^{\varepsilon})$, and in a second step we show that every weak limit is an energy solution. Using the uniqueness of energy solutions, we therefore
obtain the convergence of $(u^{\varepsilon})$.

\paragraph{Tightness} Let $(u^\varepsilon)$ solve~\eqref{eq:HQ-transformed} and write $\tilde{F}(x) = F(x) - c_1(F) x$. To prove the tightness of $(u^{\varepsilon})$ it
suffices to show that for all $\ell \in \mathbb{Z}$ the complex-valued
process $(\langle u^{\varepsilon}, e_{- \ell} \rangle)$ is tight and satisfies
a polynomial bound in $\ell$, uniformly in $\varepsilon$. We decompose
$\langle u^{\varepsilon}_t, e_{- \ell} \rangle$ as
\begin{align}\label{eq:ueps-decomp} \nonumber
   \langle u^{\varepsilon}_t, e_{- \ell} \rangle & = \langle u^{\varepsilon}_0,
   e_{- \ell} \rangle + \int_0^t \langle u^{\varepsilon}_s, \Delta e_{- \ell}
   \rangle \mathd s - \int_0^t \langle \varepsilon^{- 1} \Pi_0^N \tilde{F}
   (\varepsilon^{1 / 2} u^{\varepsilon}_s), \partial_x e_{- \ell} \rangle
   \mathd s \\ \nonumber
   &\qquad - \int_0^t \langle \partial_x \Pi_0^N \xi_s, \partial_x e_{- \ell}
   \rangle \mathd s \\
   &   = : \langle u^{\varepsilon}_0, e_{- \ell} \rangle +
  \langle S^{\varepsilon}_t, e_{- \ell} \rangle + \langle A^{\varepsilon}_t,
  e_{- \ell} \rangle + \langle M^{\varepsilon}_t, e_{- \ell} \rangle,
\end{align}
where $S^{\varepsilon}$, $A^{\varepsilon}$, $M^{\varepsilon}$ stand for
symmetric, antisymmetric and martingale part, respectively, and we
show tightness for each term on the right hand side separately. The
convergence of $\langle u^{\varepsilon}_t, e_{- \ell} \rangle$ at a fixed time
(in particular $t = 0$) follows from the fact that the law of
$u^{\varepsilon}_t$ is that of ${\mu}^{\varepsilon}$ for all $t$, and
$({\mu}^{\varepsilon})$ obviously converges to the law of the white noise
as $\varepsilon \rightarrow 0$. The linear term is tight because
\begin{align*}
   \mathbb{E} \left[ \left| \int_s^t \langle u^{\varepsilon}_r, \Delta
   e_{\ell} \rangle \mathd r \right|^p \right] & \leqslant | t - s |^{p - 1}
   \int_s^t \mathbb{E} [| \langle u^{\varepsilon}_r, \ell^2 e_{\ell} \rangle
   |^p] \mathd r \\
   & \lesssim | t - s |^{p - 1} \int_s^t \mathbb{E} [| \langle
   v^{\varepsilon}_r, \ell^2 e_{\ell} \rangle |^2]^{p / 2} \mathd r  = | t - s |^p | \ell |^{2 p} .
\end{align*}
The martingale term is for all $\varepsilon$ a mollified space--time white
noise, so its convergence is immediate.

Only the nonlinear contribution to the dynamics is nontrivial to control. Here we use the Boltzmann--Gibbs principle Proposition~\ref{prop:BG}
to get
\[ \mathbb{E} \left[ \left| \int_s^t \langle \varepsilon^{- 1} \Pi_0^N \tilde{F}
   (\varepsilon^{1 / 2} u^{\varepsilon}_s), \partial_x e_{- \ell} \rangle
   \mathd s \right|^2 \right] \lesssim | t - s |^{3 / 2 - \kappa} \ell^2
   \int_{\mathbb{R}} | F' (x) |^2 \nu (\mathd x), \]
from where the tightness in $C ([0, T], \mathbb{C})$ follows and also that any
limit point has zero quadratic variation.

Similarly we have for the time reversed process $\hat{u}^{\varepsilon}_t =
u^{\varepsilon}_{T - t}$
\begin{align}\label{eq:ueps-decomp-rev} \nonumber
   \langle \hat{u}^{\varepsilon}_t, e_{- \ell} \rangle & = \langle
   \hat{u}^{\varepsilon}_0, e_{- \ell} \rangle + \int_0^t \langle
   \hat{u}^{\varepsilon}_s, \Delta e_{- \ell} \rangle \mathd s + \int_0^t
   \langle \varepsilon^{- 1} \Pi_0^N F (\varepsilon^{1 / 2}
   \hat{u}^{\varepsilon}_s), \partial_x e_{- \ell} \rangle \mathd s \\ \nonumber
   &\qquad - \int_0^t \langle \partial_x \Pi_0^N \hat{\xi}_s, \partial_x e_{- \ell} \rangle
   \mathd s \\
  & = : \langle \hat{u}^{\varepsilon}_0, e_{- \ell}
  \rangle + \langle \hat{S}^{\varepsilon}_t, e_{- \ell} \rangle + \langle
  \hat{A}^{\varepsilon}_t, e_{- \ell} \rangle + \langle
  \hat{M}^{\varepsilon}_t, e_{- \ell} \rangle,
\end{align}
and the same arguments as before show that each term on the right hand side is
tight in $C ([0, T], \mathbb{C})$, satisfies a uniform polynomial bound, and
that any limit point of $\langle \hat{A}^{\varepsilon}, e_{- \ell} \rangle$
has zero quadratic variation. Since we have suitable moment bounds for
each term, we actually get the joint tightness:

\begin{lemma}
  \label{lem:tightness}Consider the decomposition~(\ref{eq:ueps-decomp}),
  (\ref{eq:ueps-decomp-rev}). Then the tuple
  \[ (u^{\varepsilon}_0, \hat{u}^{\varepsilon}_0, S^{\varepsilon},
     \hat{S}^{\varepsilon}, A^{\varepsilon}, \hat{A}^{\varepsilon},
     M^{\varepsilon}, \hat{M}^{\varepsilon}) \]
  is tight in $\left( \cS' \right)^2 \times C \left( [0, T], \cS' \right)^6$.
  For every weak limit $(u_0, \hat{u}_0, S, \hat{S}, \mathcal{A},
  \hat{\mathcal{A}}, M, \hat{M})$ and any $\varphi \in C^{\infty}
  (\mathbb{T})$ the processes $\langle \mathcal{A}, \varphi \rangle$ and
  $\langle \hat{\mathcal{A}}, \varphi \rangle$ have zero quadratic variation
  and satisfy $\hat{\mathcal{A}}_t = - (\mathcal{A}_T -\mathcal{A}_{T - t})$.
  Moreover, $u_t = u_0 + S_t + A_t + M_t$, $t \in [0, T]$, is for every fixed
  time a spatial white noise.
\end{lemma}

\paragraph{Convergence}Recall the definition of energy solutions to the
stochastic Burgers equation~\cite{GJ13}:

\begin{definition}
  (Controlled process)
  
  Denote with $\mathcal{Q}$ the space of continuous stochastic processes $(u, \mathcal{A})$ on $[0,T]$ with values in $\cS'$ such that
  \begin{itemize}
    \item[i)] the law of $u_t$ is the white noise ${\mu}$ for all $t \in
    [0, T]$;
    
    \item[ii)] For any test function $\varphi \in \cS$ the process $t \mapsto
    \langle \mathcal{A}_t, \varphi \rangle$ is almost surely of zero quadratic
    variation, $\langle \mathcal{A}_0, \varphi \rangle = 0$ and the pair
    $(\langle u, \varphi \rangle, \langle \mathcal{A}, \varphi \rangle)$
    satisfies the equation
    \begin{equation}
      \label{eq:controlled-decomposition} \langle u_t, \varphi \rangle =
      \langle u_0, \varphi \rangle + \int_0^t \langle u_s, \Delta \varphi
      \rangle \mathd s + \langle \mathcal{A}_t, \varphi \rangle - \langle M_t,
      \partial_x \varphi \rangle
    \end{equation}
    where $(\langle M_t, \partial_x \varphi \rangle)_{0 \leqslant t \leqslant
    T}$ is a martingale with respect to the filtration generated by $(u,
    \mathcal{A})$ with quadratic variation $[\langle M_t, \partial_x \varphi
    \rangle]_t = 2 t \| \partial_x \varphi \|_{L^2 (\mathbb{T})}^2$;
    
    \item[iii)] the reversed processes $\hat{u}_t = u_{T - t}$,
    $\hat{\mathcal{A}}_t = - (\hat{\mathcal{A}}_T -\mathcal{A}_{T - t})$
    satisfy the same equation with respect to their own filtration (the
    backward filtration of $(u, \mathcal{A})$).
  \end{itemize}
\end{definition}

The pair $(u, \mathcal{A})$ is called \tmtextit{controlled} since for
$\mathcal{A} \equiv 0$ we simply get the Ornstein--Uhlenbeck process, so in
general $u$ is a ``zero quadratic variation perturbation'' of that process.
Using the It{\^o} trick, it is not hard to show that for controlled processes
the Burgers nonlinearity is well defined:

\begin{lemma}[\cite{GJ13}, Lemma~1]
  
  Assume that $(u, \mathcal{A}) \in \mathcal{Q}$ and set for $M \in
  \mathbb{N}$
  \[ \langle \mathcal{B}^M_t, \varphi \rangle = - \int_0^t \langle (\Pi^M_0
     u_s)^2, \partial_x \varphi \rangle \mathd s. \]
  Then $(\mathcal{B}^M_t)$ converges in probability in $C \left( [0, T], \cS'
  \right)$ and we denote the limit by
  \[ \langle \int_0^t \partial_x u_s^2 \mathd s, \varphi \rangle . \]
\end{lemma}

A controlled process $(u, \mathcal{A})$ is a solution
to the stochastic Burgers equation
\[ \partial_t u = \Delta u + c \partial_x u^2 + \xi \]
if $\mathcal{A}= c \int_0^t \partial_x u_s^2 \mathd s$. According to \cite[Theorem~2]{GP15}, there is a unique energy solution.
The following theorem
thus implies our main result, Theorem~\ref{thm:main-result}.

\begin{theorem}
  Let $(u, \mathcal{A})$ be as in Lemma~\ref{lem:tightness}. Then $(u,
  \mathcal{A}) \in \mathcal{Q}$ and $u$ is the unique energy solution to
  \[ \partial_t u = \Delta u + c_2 (F) \partial_x u^2 + \xi . \]
\end{theorem}

\begin{proof}
  The tuple $(u^{\varepsilon}_0, \hat{u}^{\varepsilon}_0, S^{\varepsilon},
  \hat{S}^{\varepsilon}, A^{\varepsilon}, \hat{A}^{\varepsilon},
  M^{\varepsilon}, \hat{M}^{\varepsilon})$ converges along a subsequence
  $\varepsilon_n \rightarrow 0$, \ but to simplify notation we still
  denote this subsequence by the same symbol. Since $(u^{\varepsilon}_0,
  S^{\varepsilon}, A^{\varepsilon}, M^{\varepsilon})$ converges jointly and
  for every fixed $\varepsilon$ the process $u^{\varepsilon}$
  solves~(\ref{eq:HQ-start}), we get for $\varphi \in C^{\infty}
  (\mathbb{T})$
  \[ \langle u_t, \varphi \rangle = \langle u_0, \varphi \rangle + \langle
     S_t, \varphi \rangle + \langle \mathcal{A}_t, \varphi \rangle + \langle
     M_t, \varphi \rangle, \]
  and since $\langle S_t^{\varepsilon}, \varphi \rangle = \int_0^t \langle
  u^{\varepsilon}_s, \Delta \varphi \rangle \mathd s$ also $\langle S_t,
  \varphi \rangle = \int_0^t \langle u_s, \Delta \varphi \rangle \mathd s$.
  The same argument works for the backward process, so that $(u, \mathcal{A})
  \in \mathcal{Q}$. It remains to show that $\mathcal{A}= c_2 (F) \partial_x
  u^2$, which follows from the Boltzmann--Gibbs principle,
  Proposition~\ref{prop:BG}. For all $\varepsilon > 0$ and $M \leqslant N = 1
  / (2 \varepsilon)$
  \[ \mathbb{E} \left[ \left| \int_s^t \langle A^{\varepsilon}_r - c_2 (F)
     \partial_x (\Pi_0^M u^{\varepsilon}_r)^2, e_{- \ell} \rangle \mathd r
     \right|^2 \right] \lesssim | t - s | \ell^2 (M^{- 1} + \varepsilon \log^2
     N) \int_{\mathbb{R}} | F' (x) |^2 \nu (\mathd x), \]
  so by Fatou's lemma
  \begin{align*}
     \mathbb{E} \left[ \left| \int_s^t \langle \mathcal{A}- c_2 (F)
     \partial_x (\Pi_0^M u_r)^2, e_{- \ell} \rangle \mathd r \right|^2 \right]
     & \leqslant \liminf_{\varepsilon \rightarrow 0} \mathbb{E} \left[ \left|
     \int_s^t \langle A^{\varepsilon}_r - c_2 (F) \partial_x (\Pi_0^M
     u^{\varepsilon}_r)^2, e_{- \ell} \rangle \mathd r \right|^2 \right] \\
     & \lesssim | t - s | \ell^2 M^{- 1} \int_{\mathbb{R}} | F' (x) |^2 \nu
     (\mathd x) .
  \end{align*}
  It now suffices to send $M \rightarrow \infty$.
\end{proof}

\end{document}